\documentclass[12pt]{amsart}
\usepackage{amssymb}
\usepackage[mathscr]{eucal}
\usepackage{xspace}
\usepackage[latin1]{inputenc}

\newcommand{\R}{\mathbb{R}}

\newcommand{\ve}{\varepsilon}

\newcommand{\supp}{\text{\rm supp}\,}
\newcommand{\re}{\mathbb{R}}

\newcommand{\FSM}{\operatorname{FS}}
\newcommand{\FS}{\FSM}

\renewcommand{\subset}{\subseteq}

\newcommand{\BUC}{\operatorname{BUC}}

\newcommand{\AP}{\operatorname{AP}}
\newcommand{\WAP}{\operatorname{WAP}}

\newcommand{\BB}{{\mathcal B}}
\renewcommand{\AA}{{\mathcal A}}

\renewcommand{\supp}{\operatorname{supp}}

\theoremstyle{plain}
\newtheorem{theorem}{Theorem}[section]

\newtheorem{lemma}{Lemma}[section]

\theoremstyle{definition}

\theoremstyle{remark}

\numberwithin{equation}{section}

\begin{document}

\title[Ergodic Functions]
{Ergodic Functions That are not Almost Periodic Plus $L^1-$Mean Zero }
\author{Jean Silva}
\address{Departamento de Matem\'atica, Universidade Federal de Minas Gerais}
\email{jean@mat.ufmg.br}

\keywords{ergodic functions, almost-periodic functions, weakly almost periodic functions, algebra with mean 
value, homogenization.}
\subjclass{Primary: 35B40, 35B35, 74Q10; Secondary: 35L65, 35K55, 35B27}
\date{}

\begin{abstract}
Ergodic Functions are bounded uniformly continuous $(\BUC)$ functions that are typical realizations of continuous stationary ergodic process. 
A natural question is whether such functions are always the sum of an almost periodic with an $L^1-$mean zero $\BUC$ function. 
The paper answers this question presenting a framework that can provide infinitely many ergodic functions that are not almost periodic plus $L^1-$
mean zero.
\end{abstract}

\maketitle

\section{Introduction} \label{S:0}
Let $(\Omega,\mathcal{A},\mu)$ be a probability space and $T:\R^n\times\Omega\to\Omega$ a family of mappings (which we shall call dynamical system) with the followings properties:
\begin{enumerate}
\item[$(T_1)$](Group Property) $T(0)=Id$, $T(x+y)=T(x)\circ T(y)$, where $Id:\Omega\to\Omega$ is the identity mapping.
\item[$(T_2)$](Invariance) For every $x\in\R^n$ and every set $E\in\mathcal{A}$, we have 
$$
\text{$T(x)E\in\mathcal{A}$ and $\mu(T(x)E)=\mu(E)$}.
$$
\item[$(T_3)$](Measurability) For any measurable function $f:\Omega\to\R$, the function $f(T(x)\omega)$ defined on the cartesian product 
$\R^n\times\Omega$ is also measurable. 
\end{enumerate}
We shall say that the dynamical system $T$ is ergodic if  for any measurable function $f:\Omega\to\R$ satisfying $f(T(x)\omega)=f(\omega)$ 
for any $x\in\R^n$ and $\mu-$almost everywhere $\omega\in\Omega$ we must have $f$ is constant $\mu-$almost everywhere in $\Omega$. It is 
well known that this notion of ergodicity is equivalent to the property: If $E\in\mathcal{A}$ 
satisfies $T(x)E=E$ for all $x\in\R^n$, then $\mu(E)\in\{0,1\}$.

A measurable function $F:\R^n\times\Omega\to\R$ is a stationary ergodic process if, for some measurable function 
$f:\Omega\to\R$ and some ergodic dynamical system $T:\R^n\times\Omega\to\Omega$, we have 
$$
F(x,\omega)=f(T(x)\omega).
$$
For each fixed $\omega\in\Omega$, we call $f(T(x)\omega)$ a realization of the process $F(x,\omega)=f(T(x)\omega)$.

In the case where $\Omega$ is a compact topological space endowed with a probability measure defined in the 
Borel subsets of $\Omega$, and the dynamical system $T:\R^n\times\Omega\to\Omega$ is a continuous mapping, and moreover 
$f:\Omega\to\R$ is a continuous function, it was proven in~\cite{AH2} that, for a.a. fixed $\omega\in\Omega$, the realization 
$f(T(\cdot)\omega)$ belongs to an ergodic algebra, a concept whose definition we recall subsequently. The validity of this fact 
for general stationary ergodic process was asserted in~\cite{JKO} without proof and precise hypotheses. Nevertheless, in~\cite{Frid}, 
it is shown that for a general stationary ergodic process, if the family of realizations $\Big\{f(T(\cdot)\omega);\,\omega\in\Omega\Big\}$ 
is equicontinuous, then it is possible to reduce this case to the just mentioned case addressed in~\cite{AH2}, thus proving the validity 
of assertion in~\cite{JKO} also in this more general case, that is, for almost all fixed $\omega\in\Omega$, the realization 
$f(T(\cdot)\omega)$ belongs to an ergodic algebra. The latter is a concept introduced by Zhikov and Krivenko in~\cite{ZK} 
(see also~\cite{JKO}). In order to recall its definition, we first need to recall the definition of algebra with mean value 
(w.m.v. for short). The latter is a closed linear subspace $\AA$ of the space of bounded uniformly continuous function in 
$\R^n$ such that: a) $\AA$ is an algebra of functions; b) if $f\in\AA$, then $f(\cdot+k)\in\AA$ for all $k\in\R^n$; c) all elements of $\AA$ has a mean value, that is, 
if $f\in\AA$, the sequence $\{f(\cdot/\ve){\}}_{\ve>0}$ converges, in the duality with $L^\infty$ and compactly supported functions, to the constant $M(f)$, and, in particular, 
$$
M(f)=\lim_{R\to\infty}\frac{1}{|B(x_0,R)|}\int_{B(x_0,R)}f(y)\,dy,
$$
for $x_0\in\R^n$, where $B(x_0,R)$ is the open ball centred at $x_0$ and $|B(x_0,R)|$ denotes its $n-$dimensional Lebesgue measure. Given an algebra w.m.v. $\AA$, 
we consider the semi-norm $[f{]}_2:=M(f^2)^{1/2}$, take the quotient with respect to the equivalence relation $f\sim g \iff [f-g{]}_2=0$, and take the 
completion of the quotient space and denote it by $\BB^2$, the Besicovitch space with exponent $2$ associated with $\AA$. An algebra w.m.v. 
$\AA$ is said to be ergodic if whenever $f\in\BB^2$ satisfies $f(\cdot+k)=f(\cdot)$ in the sense of $\BB^2$ for all $k\in\R^n$, then $f$ is equivalent in $\BB^2$ to 
a constant.

A bounded uniformly continuous function over $\R^n$ is said to be an ergodic function if it belongs to some ergodic algebra w.m.v. $\AA$.
It is worth mentioning that as far as the author could verify in the literature, all examples of ergodic functions presented in the literature are: 
\begin{enumerate}
\item The continuous periodic functions.
\item The almost periodic functions. Given a continuous function $f:\R^n\to\R$ and $\epsilon>0$, we say that $p\in\R^n$ is a $\epsilon-$almost period of $f$ if 
$$
\Big|f(x+p)-f(x)\Big|<\epsilon\quad\text{for all $x\in\R^n$}.
$$
We shall denote the set of $\epsilon-$almost periods of $f$ by $\mathcal{T}(\epsilon,f)$. A continuous functions $f:\R^n\to\R$ is said to be almost periodic if for 
any $\epsilon>0$, the set $\mathcal{T}(\epsilon,f)$ is relatively dense in $\R^n$, that is, there exists $l=l(\epsilon)>0$ such that any cube with side length $l$ 
contains at least one $\epsilon-$period. The set of almost periodic functions on $\R^n$ is denoted by $\AP(\R^n)$. The following important characterizations of 
an almost periodic function are classical and can be found in \cite{B,LZ}.
A continuous function $f$ is in $\AP(\R^n)\iff$ the family of its translates 
$\{f(\cdot+t): t\in\re^n\}$ is pre-compact in the norm of sup$\iff$ $f$ may be uniformly approximated by finite linear combinations of functions in the set $\{\sin(\lambda\cdot x),\,\cos(\lambda\cdot x){\}}_{\lambda\in\re^n}$.
\item The continuous functions with limit at infinite.
\item The Fourier-Stieltjes functions, which are the uniform approximations of the bounded continuous functions $f$ which satisfy $f(x)=\int_{\re^n}e^{ix\cdot y}\,d\mu(y)$, for some complex-value Radon measure in $\re^n$(cf. also~\cite{AH3}). This space is denoted by $\FS(\R^n)$.
\item The weakly almost periodic functions, whose space is denoted by $\mathcal{W}\AP(\R^n)$, is the space of the bounded continuous functions $f$ in $\R^n$, such that de family $\{f(\cdot+t): t\in\re^n\}$ is pre-compact in the weak topology of $C(\re^n)$(the space of the bounded continuous functions). The main properties of this space was studied by Erberlein in \cite{Eb1,Eb2}. 
In~\cite{Ru}, Rudin proved that the inclusion 
$\FS(\re^n)\subset\WAP(\re^n)$ is strict, showing an example of a weakly almost periodic function that cannot be approximated in the sup norm by Fourier-Stieltjes transforms.
\item The weakly* almost periodic functions over $\R^n$. 
In~\cite{Eb2}, Eberlein established the following important decomposition for functions 
$f\in\mathcal{W}\AP(\R^n)$, which allows to write any function as 
$$
f=f_{ap}+f_0,
$$
where $f_{ap}\in\AP(\R^n)$ and $M(|f_0|^2)=0$. This property satisfied by the weakly almost periodic functions served as defining property to a natural broader 
class of functions considered by H. Frid in~\cite{Frid}. This class is denoted by $\mathcal{W}^{*}\AP(\R^n)$ and it is defined as the algebraic sum 
$$
\mathcal{W}^{*}\AP(\R^n):=\AP(\R^n)+\mathcal{N}(\R^n),
$$
where $\mathcal{N}(\R^n)$ is the subspace of the bounded uniformly continuous function $f$ such that $M\left(|f|\right)=0$. Thus, it is clear that 
$\mathcal{W}\AP(\R^n)\subset\mathcal{W}^{*}\AP(\R^n)$.
\end{enumerate}

Therefore, all known examples of ergodic functions over $\R^n$ are weakly* almost periodic functions. Therefore, an important question is 
whether there exist ergodic functions which are not weakly* almost periodic functions. In this article, we shall give an answer to this question by constructing 
a probability compact space and a continuous function such that the majority of its realizations are ergodic functions that are beyond of the weakly* almost periodic settings. 
This will be done through the introduction of a family of $\BUC-$functions $\Omega$ that is equivalent, topologically and measure theoretically to 
$\Omega_0\times \R^n/{\mathbb{Z}^n}$, where $\Omega_0:=\{-1,1\}^{\mathbb{Z}^n}$, and a dynamical system over $\Omega$ that is equivalent to a natural 
shift mapping on $\Omega_0\times \R^n/{\mathbb{Z}^n}$. 
This shows not only the existence of ergodic functions with different behavior from those 
discovered by Eberlein in \cite{Eb1} but also brings out a curious statistical evidence: The weakly almost periodic oscillations may not be so often. From the practical 
point of view, what happens is that in many situations, the stochastic homogenization problems of differential type, can be 
reduced to a homogenization problem of the differential operators whose coefficients are ergodic functions. Thus, the afore-knowledge of 
the nature of the "self-averaging" behavior of these ergodic functions can dictate the complexity of the solution. For example, if you know that almost all realizations 
are of weakly* almost periodic type, an application of the Birkhoff's theorem allows us to conclude that in fact almost all realizations are of almost periodic type. The 
reduction of the stationary ergodic settings to the almost periodic one can simplify a lot the solution of the homogenization's problem(compare for instance the paper \cite{Frid} with \cite{CSW} and 
\cite{I} with \cite{So}).

\section{The construction of the Example}\label{S:2}
We construct the example in dimension $1$ in order to simplify notations. However, as the reader will see, the construction may be easily extended to the 
multi-dimensional case.

Denote the set of integers numbers by $\mathbb{Z}$. On the compact set 
$\{-1,1\}$, we define an elementary radon probability measure $\lambda=\lambda_q$ as follows: The measure of the one-point set 
$\{-1\}$  is equal to $q$, and the measure of the set $\{1\}$ is equal to $1-q$. Here, we consider $0<q<1$. Let $\Omega_0$ be the space product
$$
\Omega_0:=\{-1,1\}^{\mathbb{Z}},
$$ 
that is a compact space by Tychonoff's theorem. The elements of $\Omega_0$ are sequences which assume values in the set 
$\{-1,1\}$. Denote by $\nu=\nu_q$ the product of the elementary measures $\lambda$ and the 
function $\tau:\mathbb{Z}^n\times \Omega_0\to \Omega_0$ by 
$\tau(k,\mathbf{x}):=\{{\mathbf{x}}_{j+k}{\}}_{j\in\mathbb{Z}}$(The shift operator). For simplicity, we shall use $\tau_k{\mathbf{x}}$ to denote $\tau(k,{\mathbf{x}})$. It is 
well known that the function $\tau$ is an ergodic discrete dynamical system over the compact probability space $(\Omega_0,\nu)$ (see, e.g., \cite{RM} pag. 101)
The followings concepts will be useful for us here. 
\begin{itemize}
\item A real sequence $\{\mathbf{x}_k{\}}_{k\in\mathbb{Z}}$ is said to be a periodic sequence if there exists a integer $p>0$ such that 
$$
\mathbf{x}_{k+p}=\mathbf{x}_k,\quad \text{for all $k\in\mathbb{Z}$}.
$$
\item A real sequence $\{\mathbf{x}_k{\}}_{k\in\mathbb{Z}}$ is said to be an almost periodic sequence if to any $\epsilon>0$ there correspondes an 
integer $k_0(\epsilon)$, such that among any $k_0$ consecutive integers there exists an integer $p$(called $\epsilon-$almost period) with the property 
$$
|\mathbf{x}_{k+p}-\mathbf{x}_k|<\epsilon,\quad\text{for all $k\in\mathbb{Z}$}.
$$
\end{itemize}

The next lemma states the fact the almost periodic elements of $\Omega_0$ are periodic.
\begin{lemma}\label{quaseperiodOmega}
Every almost periodic sequence $\mathbf{x}\in\Omega_0$ is a periodic sequence.
\end{lemma}
\begin{proof}
First, suppose that $\mathbf{x}\in\Omega_0$ is an almost periodic element. Then, 
let $\epsilon\in(0,1)$ and choose a $\epsilon-$almost period $p>0$. Hence, by definition 
$$
|\mathbf{x}_{k+p}-\mathbf{x}_k|<\epsilon,\quad\text{for all $k\in\mathbb{Z}$}.
$$
Now, we have only two possibilities: 
\begin{itemize}
\item $\mathbf{x}_{k+p}=\mathbf{x}_{k}$ for all $k\in\mathbb{Z}$.
\item $\mathbf{x}_{k_0+p}\neq\mathbf{x}_{k_0}$ for some $k_0\in\mathbb{Z}$. Since our sequence assumes its values only 
in the set $\{-1,1\}$, we must have $2=|\mathbf{x}_{k_0+p}-\mathbf{x}_{k_0}|<\epsilon$, which is a contradiction with the 
choise of $\epsilon$. Thus, only the first possibility must happen. This proves our lemma.
\end{itemize}
\end{proof}

Now, take 
$\varphi\in C(\R)$ such that 
\begin{itemize}
\item $\supp \varphi\subset (-1/2,1/2)$. 
\item $0\le \varphi(x)\le 1$ for any $x\in\R$.
\item $\varphi(x)=1$ if and only if $x=0$.
\end{itemize}
The set $\Omega_0$ can be naturally associated with the set 
$$
\left\{\Lambda_{\mathbf{x}}:\R\to \R;\,\Lambda_{\mathbf{x}}(\cdot):=\sum_{m\in \mathbb{Z}}{\mathbf{x}}_m\varphi(\cdot-m),\,\mathbf{x}\in\Omega_0\right\}.
$$
The following lemma will be important for our proposals.
\begin{lemma}\label{R1}
There exists a bijection between the set
$$
\Omega:=\bigg\{\Lambda_{\mathbf{x}}(\cdot+\delta);\, (\mathbf{x},\delta)\in\Omega_0\times \R\bigg\}
$$
and the set $\Omega_0\times \R/\mathbb{Z}$.
\end{lemma}

\begin{proof}
1. First, it is important to establish the following remark: If $\Lambda_{\mathbf{x}}(t+\delta_1)=\Lambda_{\mathbf{y}}(t+\delta_2)$ for all $t\in \R$, then, the properties below hold:
\begin{itemize}
\item ${\tau}_{\left\lfloor \delta_1 \right\rfloor} \mathbf{x}=\tau_{\left\lfloor \delta_2 \right\rfloor} \mathbf{y}$.
\item $\delta_1-\left\lfloor \delta_1 \right\rfloor=\delta_2-\left\lfloor \delta_2 \right\rfloor$,
\end{itemize}
where $\lfloor x\rfloor$ denotes the unique number in $\mathbb{Z}$ such that $x-\lfloor x\rfloor \in [0,1)$. Here, we shall identify the $1-$dimensional 
torus $\R/\mathbb{Z}$ with the interval $[0,1)$. 
First, we claim that $\delta_1-\delta_2\in \mathbb{Z}$. Indeed, define 
$\theta:=\delta_1-\delta_2-\left\lfloor \delta_1-\delta_2\right\rfloor\in [0,1)$.  Since $\Lambda_{\mathbf{x}}(t+\delta_1)=\Lambda_{\mathbf{y}}(t+\delta_2)$ for all $t\in \R$, 
changing $t$ by $t-\delta_1$ and using the definition of $\Lambda_{\mathbf{x}}$ and $\Lambda_{\mathbf{y}}$, we get:
\begin{eqnarray*}
&& \sum_{m\in \mathbb{Z}}\mathbf{x}_m\varphi(t-m)=\sum_{m\in \mathbb{Z}}\mathbf{y}_m\varphi\left(t-m+\delta_2-\delta_1\right)=
\sum_{m\in \mathbb{Z}}\mathbf{y}_m\varphi\left(t-m+\theta+\left\lfloor \delta_1-\delta_2\right\rfloor\right)\\
&&\qquad\qquad =\sum_{m\in \mathbb{Z}}(\tau_{\left\lfloor \delta_1-\delta_2 \right\rfloor}\mathbf{y})_m\,\varphi\left(t-m+\theta \right),\quad \text{for all $t\in\R$}.
\end{eqnarray*}
Taking $t=0$, we have that there exists an unique $m_0\in\mathbb{Z}$ such that 
\begin{eqnarray*}
&& \mathbf{x}_0= \sum_{m\in \mathbb{Z}}\mathbf{x}_m\varphi(-m)=\sum_{m\in \mathbb{Z}}(\tau_{\left\lfloor \delta_1-\delta_2 \right\rfloor}\mathbf{y})_m\,\varphi\left(-m+\theta \right)\nonumber\\
&&\qquad\qquad 
=(\tau_{\left\lfloor \delta_1-\delta_2 \right\rfloor}\mathbf{y})_{m_0}\,\varphi\left(\theta-m_0 \right).
\end{eqnarray*}
Taking into account that $\mathbf{x}_0,(\tau_{\left\lfloor \delta_1-\delta_2 \right\rfloor}\mathbf{y})_{m_0}\in \{-1,1\}$, we must have 
$\varphi\left(\theta-m_0 \right)=1$. Thus, by the conditions on the function $\varphi$, it implies that $m_0=\theta$.  Thus, $\theta\in [0,1)\cap \mathbb{Z}$ which gives that 
$\theta=0$ and the claim is proved. 
From the claim, we have that $\delta_1-\left\lfloor \delta_1 \right\rfloor=\theta=\delta_2-\left\lfloor \delta_2 \right\rfloor$. Moreover, 
\begin{eqnarray}\label{E1}
&&\Lambda_{\mathbf{x}}(t+\delta_1)= \sum_{m\in \mathbb{Z}}\mathbf{x}_m\,\varphi\left(t+\delta_1-m\right)=\sum_{m\in \mathbb{Z}^n}\mathbf{x}_m\,\varphi\left(t+\theta+\left\lfloor \delta_1 \right\rfloor-m\right)\nonumber\\
&&\qquad\qquad
=\sum_{m\in \mathbb{Z}}(\tau_{\left\lfloor \delta_1\right\rfloor}\mathbf{x})_{m}\,\varphi\left(t+\theta-m \right),\,\text{for all $t\in\R$},
\end{eqnarray}
and the same holds if we change $\mathbf{x}$ by $\mathbf{y}$ and $\delta_1$ by $\delta_2$. Taking into account that $\Lambda_{\mathbf{x}}(t+\delta_1)=\Lambda_{\mathbf{y}}(t+\delta_2)$ for all $t\in\R$, we get:
$$
\sum_{m\in \mathbb{Z}}(\tau_{\left\lfloor \delta_1\right\rfloor}\mathbf{x})_{m}\,\varphi\left(t+\theta-m \right)=
\sum_{m\in \mathbb{Z}}(\tau_{\left\lfloor \delta_2\right\rfloor}\mathbf{x})_{m}\,\varphi\left(t+\theta-m \right),\,\text{for all $t\in\R$}.
$$
Therefore, changing $t$ by $t-\theta$ in the above equality, we have
$$
(\tau_{\left\lfloor \delta_1\right\rfloor}\mathbf{x})_{m}=(\tau_{\left\lfloor \delta_2\right\rfloor}\mathbf{y})_{m},\,\text{for all $m\in\mathbb{Z}$},
$$
which establishes the remark.

2. Define $H:\Omega\to \Omega_0\times [0,1)$ by
$$
H(\Lambda_{\mathbf{x}}(\cdot+\delta))=\left(\tau_{\left\lfloor \delta \right\rfloor}\mathbf{x},\delta-\left\lfloor \delta \right\rfloor\right).
$$
By the step 1, the function $H$ is well defined and is onto.

3. We claim that the function $H$ is one-to-one. In order to show this, let $\mathbf{x},\mathbf{y}\in\Omega_0$ and $\delta_1,\delta_2\in\R$ be such that 
$$
H\left(\Lambda_{\mathbf{x}}(\cdot+\delta_1)\right)=H\left(\Lambda_{\mathbf{y}}(\cdot+\delta_2)\right).
$$
By the definition of $H$, we have:
\begin{itemize}
\item ${\tau}_{\left\lfloor \delta_1 \right\rfloor} \mathbf{x}=\tau_{\left\lfloor \delta_2 \right\rfloor} \mathbf{y}$.
\item $\delta_1-\left\lfloor \delta_1 \right\rfloor=\delta_2-\left\lfloor \delta_2 \right\rfloor=\theta$.
\end{itemize} 
Now, using these relations in definition of $\Lambda_{\mathbf{x}}$ and the relation~$\eqref{E1}$, we have the following equality:
\begin{eqnarray*}
&&\Lambda_{\mathbf{x}}(t+\delta_1)=\sum_{m\in \mathbb{Z}}(\tau_{\left\lfloor \delta_1\right\rfloor}\mathbf{x})_{m}\,\varphi\left(t+\theta-m \right)\\
&&\qquad =\sum_{m\in \mathbb{Z}}(\tau_{\left\lfloor \delta_2\right\rfloor}\mathbf{y})_{m}\,\varphi\left(t+\theta-m \right)=\Lambda_{\mathbf{y}}(t+\delta_2),\,\text{for all $t\in\R$}.
\end{eqnarray*}

\end{proof}

By Lemma~\ref{R1}, the set $\Omega$ inherits the probabilistic and the topologicals features of the space $\Omega_0\times [0,1)$  in a natural way.  Let $\mu=\mu_q$ be 
the probability measure on $\Omega$ associated to the measure on $\Omega_0\times [0,1)$ defined as the product of measure $\nu$ on $\Omega_0$ and the Lebesgue 
measure on $[0,1)$. From now on, we shall denote by $\Omega_1$ the space $\Omega_0\times [0,1)$.
\begin{lemma}\label{R1.1}
Let $\Big(\Omega,\mu\Big)$ be the probability space constructed above and 
$T:\R\times\Omega\to\Omega$ be the shift operator defined on $\Omega$, that is, $T(z,\omega):=\omega(\cdot+z)$. Then, we have the followings properties: 
\begin{enumerate}
\item[(i)] $$
T(z)=H^{-1}\circ S(z)\circ H,
$$
where $S(z):\Omega_1\to \Omega_1$ is given by 
\begin{equation}\label{SisDin2}
S(z)(\mathbf{x},\theta):=\Big({\tau}_{\left\lfloor z+\theta \right\rfloor} \mathbf{x},z+\theta-\left\lfloor z+\theta \right\rfloor\Big)\,(z\in\R).
\end{equation}
\item[(ii)] The function $S(z):\Omega_1\to \Omega_1$ defined by \eqref{SisDin2} is an ergodic dynamical system with respect to 
the product of measure $\nu$ on $\Omega_0$ and the Lebesgue measure on $[0,1)$. 
\item[(iii)] The mapping $T:\R\times\Omega\to\Omega$ is an ergodic dynamical system. 
\end{enumerate}
\end{lemma}
\begin{proof}
1. Writing $\omega(\cdot)=\Lambda_{\mathbf{x}}(\cdot+\delta)$ for some $\mathbf{x}\in\Omega_0$ and some $\delta\in\R$, we have by definition of $H$ in 
the proof of the lemma \ref{R1}(see step 2):
$$
H(T(z)\omega)=H\left(\omega(\cdot+z)\right)=H\left(\Lambda_{\mathbf{x}}(\cdot+\delta+z)\right)=\Big(\tau_{\left\lfloor z+\delta \right\rfloor}
\mathbf{x},z+\delta-\left\lfloor z+\delta \right\rfloor\Big).
$$
On the other hand, 
\begin{eqnarray*}
&& S(z)\left(H(\omega)\right)=S(z)\left({\tau}_{\left\lfloor \delta \right\rfloor}\mathbf{x},\delta-\left\lfloor \delta \right\rfloor\right)\\
&&\quad = \Big({\tau}_{\left\lfloor z+\delta - \left\lfloor \delta \right\rfloor \right\rfloor+ \left\lfloor \delta \right\rfloor}\mathbf{x}, 
z+\delta - \left\lfloor \delta \right\rfloor-\left\lfloor z+\delta - \left\lfloor \delta \right\rfloor \right\rfloor\Big)\\
&&\quad= \Big({\tau}_{\left\lfloor z+\delta \right\rfloor}\mathbf{x}, z+\delta-\left\lfloor z+\delta \right\rfloor\Big)=H(T(z)\omega),
\end{eqnarray*}
where we have used the property $\left\lfloor t+k \right\rfloor=\left\lfloor t\right\rfloor+k$ for all $t\in\R$ and $k\in\mathbb{Z}$. This proves 
the item $(i)$.

2. By the item $(i)$, we have that $S(z)=H\circ T(z)\circ H^{-1}$. Since $T(z)$ is a shift operator, then it is clear that 
$T(0)=I_{\Omega}$ and $T(z_1+z_2)=T(z_1)\circ T(z_2)$(group property). Hence, the same happens with the function $S(z)$, that is, 
$S(z_1+z_2)=S(z_1)\circ S(z_2)$.

3. If $E\subset \Omega_1$ is a Borel set, then same occurs with $S(z)(E)$ for any $z\in\R$. This can be seen 
by noting that the collection of all sets having this property forms a $\sigma-$algebra. Moreover, taking into account that 
the measurable sets of $\Omega_0$ are sent to measurable sets by the transformation $\tau_k$ for any $k\in\mathbb{Z}$ and the 
same happens with the space $[0,1)$ with respect to transformation 
$\cdot+z-\left\lfloor \cdot+z\right\rfloor$ for any $z\in\R$, we can deduce that the rectangles are also sent to  
measurable sets of $\Omega_1$ by the mapping $S(z)$. Therefore, $S(z)(E)$ is a Borel set if $E$ is a Borel set for any $z\in\R$. 
Let $\mathcal{P}$ be the measure defined as the product of the measure $\nu$ on $\Omega_0$ with the Lebesgue measure $d\theta$ on $[0,1)$. 
We now shall show that $\mathcal{P}\left(S(z)(E)\right)=\mathcal{P}(E)$ for any Borel set $E\subset \Omega_1$ 
and any $z\in\R$. Take a Borel set $E\subset \Omega_1$ and $z\in\R$. Let $1_{E}(\cdot)$ be the 
characteristic function of the set $E$. Since $1_{E}\Big(S(-z)(\mathbf{x},\theta)\Big)=1_{S(z)(E)}(\mathbf{x},\theta)$ 
and $S(-z)(E)$ is a Borel set, then the function $(\mathbf{x},\theta)\mapsto 1_{E}\Big(S(z)(\mathbf{x},\theta)\Big)$ is 
measurable. Hence, by the Fubini Theorem
\begin{eqnarray*}
&&\mathcal{P}\Big(S(z)(E)\Big)=\int_{\Omega_0\times[0,1)}1_{S(z)(E)}(\mathbf{x},\theta)\,d\mathcal{P}(\mathbf{x},\theta)=
\int_{\Omega_0\times[0,1)}1_{E}\Big(S(-z)(\mathbf{x},\theta)\Big)\,d\mathcal{P}(\mathbf{x},\theta)\\
&&\qquad=\int_{\Omega_0\times[0,1)}1_{E}\Big({\tau}_{\left\lfloor \theta - z \right\rfloor}{\mathbf{x}},
\theta-z-\left\lfloor\theta -z \right\rfloor\Big)\,d\mathcal{P}(\mathbf{x},\theta)\\
&&\qquad=\int_{[0,1)}\Bigg\{\int_{\Omega_0}
1_{E}\Big({\tau}_{\left\lfloor \theta - z \right\rfloor}{\mathbf{x}},
\theta-z-\left\lfloor\theta -z \right\rfloor\Big)\,d\nu(\mathbf{x})\Bigg\}\,d\theta\\
&&\qquad=\int_{[0,1)}\Bigg\{\int_{\Omega_0}1_{E}\Big(\mathbf{x},
\theta-z-\left\lfloor\theta -z \right\rfloor\Big)\,d\nu(\mathbf{x})\Bigg\}\,d\theta\\
&&\qquad=\int_{\Omega_0\times[0,1)}1_E(\mathbf{x},\theta)\,d\nu(\mathbf{x})d\theta=\mathcal{P}(E).
\end{eqnarray*}

4. We claim that the function $(z,\mathbf{x},\theta)\mapsto f\left(S(z)(\mathbf{x},\theta)\right)$ 
defined on the cartesian product $\R\times\Omega_1$ is measurable for 
any measurable function $f:\Omega_1\to\R$. We can prove this claim reasoning as follows: 
Using an approximation argument, it is enough to consider $f(\cdot)=1_E(\cdot)$, where $E$ is a 
measurable set of $\Omega_1$ and $1_E(\cdot)$ is the characteristic function of the set $E$. Then, 
we observe that the class of all measurable sets $E\subset\Omega_1$ such that the function 
$(z,\mathbf{x},\theta)\mapsto 1_{E}\left(S(z)(\mathbf{x},\theta)\right)$ is measurable on 
$\R\times\Omega_1$ is a $\sigma-$algebra. Finally, we will finish the proof of the 
claim by showing that the rectangles of $\Omega_1$ belongs to this $\sigma-$algebra. For this, 
write $E=E_1\times E_2$, where $E_1\subset \Omega_0$ and $E_2\subset [0,1)$ are measurable sets. Note that 
$$
1_{E}\left(S(z)(\mathbf{x},\theta)\right)=1_{E_1}\left({\tau}_{\left\lfloor \theta + z \right\rfloor}\mathbf{x}\right)\,
1_{E_2}\left(z+\theta-\left\lfloor z+\theta \right\rfloor\right)=f_1\left(z,\theta,\mathbf{x}\right)\,f_2\left(z,\theta,\mathbf{x}\right),
$$
with obvious notations for $f_1, f_2$.

Since the values of the function $f_1$ are in the set $\{0,1\}$, it is sufficient to verify that $f_1^{-1}(1)$ is a measurable set 
in $\R\times[0,1)\times\Omega_0$. Note that $(z,\theta,\mathbf{x})\in f_1^{-1}(1)$ if and only if 
$\mathbf{x}\in {\tau}_{-\left\lfloor \theta + z \right\rfloor}\left(E_1\right)$. Furthermore, 
the value of $\left\lfloor \theta + z \right\rfloor$ can be $\left\lfloor z \right\rfloor$ or 
$\left\lfloor z \right\rfloor+1$ depending on the case if the sum of the fractional part of $z$ with 
$\theta$ is smaller or grater than $1$. Due to this, we partitioned the set $\R\times[0,1)$ as follows:
$$
\R\times[0,1)=\Big(\bigcup_{k\in\mathbb{Z}}V+(k,0)\Big)\bigcup \Big(\bigcup_{k\in\mathbb{Z}}V^c+(k,0)\Big),
$$
where $V:=\{(z,\theta)\in [0,1)^2;\,z+\theta<1\}$ and $V^c:=\{(z,\theta)\in [0,1)^2;\,z+\theta\ge 1\}$. Thus, 
taking into account this decomposition, it is easily seen that 
$$
f_1^{-1}(1)=\Big(\bigcup_{k\in\mathbb{Z}}\{V+(k,0)\}\times \tau_{-k}(E_1)\Big)\bigcup
\Big(\bigcup_{k\in\mathbb{Z}}\{V^c+(k,0)\}\times \tau_{-k-1}(E_1)\Big),
$$
which provides the measurability of the function $f_1$. Similar procedures can be made 
for the function $f_2$. This completes the proof of the claim.

5. In this step, we shall show that the dynamical system $\{S(z){\}}_{z\in\R}$ is ergodic. Let 
$f:\Omega_1\to\R$ be an invariant function, that is, 
$f\left(S(z)(\mathbf{x},\theta)\right)=f(\mathbf{x},\theta)$ for all $z\in\R$ and for 
$\mathcal{P}-$almost everywhere $(\mathbf{x},\theta)\in\Omega_1$. By definition 
of $S(z)$, we have 
$$
f\left({\tau}_{\left\lfloor z+\theta \right\rfloor} \mathbf{x},z+\theta-\left\lfloor z+\theta \right\rfloor\right)=f(\mathbf{x},\theta),
\quad\text{$\mathcal{P}-$almost everywhere $(\mathbf{x},\theta)$ and all $z\in\R$}.
$$
Taking $z=k$, we obtain
$$
f({\tau}_k\mathbf{x},\theta)=f(\mathbf{x},\theta),\quad\text{$\mathcal{P}-$almost everywhere $(\mathbf{x},\theta)$ and all $k\in\mathbb{Z}$}.
$$
Since the mapping $\{\tau{\}}_{k\in\mathbb{Z}}$ is ergodic, we deduce that the function $f$ does not depend of the first variable. Using this 
in the second equation above, we have
$$
f\left(z+\theta-\left\lfloor z+\theta \right\rfloor\right)=f(\theta),\quad\text{for all $z\in\R$ and almost everywhere $\theta\in[0,1)$}.
$$
As the mapping $(z,\theta)\mapsto z+\theta-\left\lfloor z+\theta \right\rfloor\in [0,1)$ is ergodic, we see that the function $f$ also is  
independent of the second variable. Thus, $f$ is equivalent to a constant. The item $(iii)$ is a direct consequence of the itens $(i)$ and 
$(ii)$. Hence, we finish the proof of the lemma.
\end{proof}

To sum up, the elements of the space $\Omega$ are uniformly continuous functions $\omega$ defined in $\R$ and the shift operator 
$T(z,\omega):=\omega(\cdot+z)$ is an ergodic dynamical system by the Lemma \ref{R1.1}. Therefore, by the arguments in the proof of the Theorem $3.1$ of 
\cite{AH2}, it follows that given any continuous function $f:\Omega\to\R$, for a.a. $\omega\in\Omega$, $f\left(T(\cdot)\omega\right)$ belongs to an ergodic algebra. 
Defining the function $f:\Omega\to \R$ by $f(\omega)=\omega(0)$, it is easy to see that $f$ is 
continous on $\Omega$ and its realization by the dynamical system $T$ satisfies $f\left(T(\cdot)\omega\right)=\omega(\cdot)$. Hence, we conclude that almost all elements of 
$\Omega$ are ergodic functions. Therefore, at this point, a natural question arise: What is the amount of the functions in the probability space $\Omega$ that are periodic, almost-periodic, 
weakly almost periodic or more generally almost periodic plus $L^1-$ mean zero?  
The aim of the next lemma is to reckon the amount of the periodic functions in $\Omega$. 
\begin{lemma}\label{R2}
The $\mu$-measure of the set 
$$
\bigg\{\omega\in\Omega;\,\text{$\omega(\cdot)$ is a periodic function}\bigg\}
$$
is null. 
\end{lemma}
\begin{proof}
1. Let $\omega\in\Omega$ be a periodic function. Thus, there exists a sequence $\mathbf{x}\in\Omega_0$, $\delta\in\R$ and a number 
$p\in {\R}_{+}$(we can assume that 
$p>1$) such that $\omega(\cdot)=\Lambda_{\mathbf{x}}(\cdot+\delta)$ and 
$$
\Lambda_{\mathbf{x}}(\cdot+k\, p+\delta)=\Lambda_{\mathbf{x}}(\cdot+\delta),
\quad\text{on $\R$ and for all $k\in \mathbb{Z}$}.
$$
Due to the identification given by lemma~\ref{R1}, we must have:
\begin{itemize}
\item ${\tau}_{\left\lfloor p\,+\delta \right\rfloor} \mathbf{x}=\tau_{\left\lfloor \delta \right\rfloor}\mathbf{ x}$.
\item $p\,+\delta-\left\lfloor p+\delta\right\rfloor=\delta-\left\lfloor \delta \right\rfloor$.
\end{itemize}
Hence, from the second relation, we deduce that $p\in\mathbb{Z}$ and from the first we get 
$$
\tau_{p}\left(\tau_{\left\lfloor \delta \right\rfloor}\mathbf{x}\right)=\tau_{\left\lfloor \delta \right\rfloor}\mathbf{x},
$$
that is, the sequence $\tau_{\left\lfloor \delta \right\rfloor}\mathbf{x}$ is also periodic with period 
$[0,p)\cap \mathbb{Z}$. Thus,
$$
\bigg\{\omega\in\Omega;\,\text{$\omega(\cdot)$ is a periodic function}\bigg\}=\Omega_0^{\rm{Per}}\times [0,1),
$$
where $\Omega_0^{\rm{Per}}:=\left\{\mathbf{x}\in\Omega_0;\,\text{$\mathbf{x}$ is a periodic sequence}\right\}$. Then, it is enough to prove that the set 
$\Omega_0^{\rm{Per}}$ has null measure. 

2. We claim that the set $\Omega_0^{\rm{Per}}$ is at most countable.  Indeed, observe that 
$$
\Omega_0^{\rm{Per}}=\cup_{p\in {\mathbb{Z}}_{+}}\big\{\mathbf{x}\in\Omega_0;\,\tau_{p} \mathbf{x}=\mathbf{x}\big\}.
$$

Moreover, given $p\in {\mathbb{Z}}_{+}$ define 
$$
\mathcal{C}:=\{-1,1\}^{\mathbb{Z}\cap [0,p)},
$$
be the set of finite sequence $\alpha=(\alpha_m)$ with $m\in\mathbb{Z} \cap [0,p)$ which assumes its values in the set $\{-1,1\}$. Hence, we can note that 
$$
\big\{x\in\Omega_0;\,\tau_{p} \mathbf{x}=\mathbf{x},\quad \big\}=\cup_{\alpha\in {\mathcal{C}}}\big\{x^{\alpha}\big\},
$$
where $x^{\alpha}\in \Omega_0$ is such that $x_m^{\alpha}=\alpha_{\overline{m}^p}$.  Since $\sharp \mathcal{C}=2^{p}$, the claim is verified. 

3.  The proof of the lemma is completed by noting that the measure $\nu$ attributes value zero to any point in $\Omega_0$, which implies that any countable set in 
$\Omega_0$ has $\nu$-measure zero. 

\end{proof}

In the next Lemma, we analyze the amount of elements in the set $\Omega$ that are almost-periodic functions.  

\begin{lemma}\label{R3}
The $\mu$-measure of the set 
$$
\bigg\{\omega\in\Omega;\,\text{$\omega(\cdot)$ is an almost-periodic function}\bigg\}
$$
is null. 
\end{lemma}
\begin{proof}
First, remember that a function $\omega\in C(\R)$ is said to be an almost-periodic function if only if the set $\big\{\omega(\cdot+T){\big\}}_{T\in\R}$ is strongly pre-compact 
in $C(\R)$. Let $\omega\in\Omega$ be an almost-periodic function. By definition of $\Omega$, there exist a sequence $\mathbf{x}\in\Omega_0$ and $\delta\in\R$ such that 
$$
\omega(\cdot)=\Lambda_{\mathbf{x}}(\cdot+\delta)=\sum_{m\in \mathbb{Z}}{\mathbf{x}}_m\,\varphi(\cdot+\delta-m).
$$
Since $\omega$ is an almost-periodic function, the set $\big\{\omega(\cdot+k){\big\}}_{k\in\mathbb{Z}}$ is strongly pre-compact in $C(\R)$. Thus, there exists a 
subsequence $\{k_j{\}}_{j\ge 1}$ such that the sequence o functions $\{\omega(\cdot+k_j){\}}_{j\ge 1}$ converges uniformly in $\R$ as $j\to \infty$. Since  
$$
\omega(\cdot+k_j)=\sum_{m\in \mathbb{Z}}{\mathbf{x}}_{m+k_j}\,\varphi(\cdot+\delta-m),
$$
we have $\{\tau_{k_j}{\mathbf{x}}{\}}_{j\ge 1}=\{\omega(\cdot-\delta+k_j){\}}_{j\ge1}$ converges uniformly in $\mathbb{Z}$. Thus, the sequence $\mathbf{x}$ is an almost-periodic sequence. By the Lemma 
\ref{quaseperiodOmega}, it follows that it must be a periodic sequence. Taking into account the $\nu$-negligibleness of the set of periodic sequence 
in $\Omega_0$(see the previous lemma), we conclude the proof of the lemma.
\end{proof}

In \cite{Frid}, the following theorem was established as the crucial point in the proof of the main theorem therein. 
\begin{theorem}[see~\cite{Frid}]\label{PaperHermano}
Let $\Big(\mathcal{E},\mathbb{P}\Big)$ be a probability space and $\mathcal{T}:\R^n\times \mathcal{E}\to\mathcal{E}$ an 
ergodic dynamical system.  Let 
$\gamma:\mathcal{E}\to\R$ be a measurable function and define $F(\cdot,\omega):=\gamma(\mathcal{T}(\cdot)\omega)$. Suppose that 
we have the following:
\begin{enumerate}
\item [(F0)] The family $\Big\{F(\cdot,\omega);\,\omega\in\mathcal{E}\Big\}$ is equicontinuous. 
\item[(F1)] For $\mathbb{P}-$a.a. $\omega\in\mathcal{E}$, 
$$
F(\cdot,\omega)\in\mathcal{W}^*\AP(\R^n).
$$
\end{enumerate}
Then, for $\mathbb{P}-$a.a. $\omega\in\mathcal{E}$, $F(\cdot,\omega)\in\AP(\R^n)$.
\end{theorem}
\begin{proof}[Sketch of the proof]\footnote{The author thanks H.~Frid for providing him the sketch of the proof presented here.}
The first and crucial step of the proof is to find a suitable way to extract the almost periodic component from $F(\cdot,\omega)$, 
for each $\omega\in\mathcal{E}$, that is, since $F(\cdot,\omega)=F_{ap}(\cdot,\omega)+F_{\mathcal{N}}(\cdot,\omega)$, with 
$F_{ap}(\cdot,\omega)\in\AP(\R^n)$ and $F_{\mathcal{N}}(\cdot,\omega)\in\mathcal{N}(\R^n)$ for a.a. $\omega\in\mathcal{E}$, we need 
to devise a way to obtain $F_{ap}(\cdot,\omega)$ from $F(\cdot,\omega)$ such that $F_{ap}:\R^n\times\mathcal{E}\to\R$ is a stationary 
ergodic process with $F_{ap}(y,\omega)=\tilde{\gamma}(\mathcal{T}(y)\omega)$, where $\tilde{\gamma}(\omega)=F_{ap}(0,\omega)$. 
The way to process this extraction of that almost periodic component is by using the approximation of the identity $\phi_{\alpha}$ which 
is known to exist, both from classical Bochner-Fej\'er polynomials(see, e.g.,\cite{B}) and from the fact that the Bohr compact is a topological 
group(see, e.g., \cite{HewittRoss}) for which the existence of such approximation is known (see~\cite{HewittRoss}) and it is a generalized 
sequence, or net, in $\AP(\R^n)$. Using the approximation of the identity we have 
$F(\cdot,\omega)*\phi_{\alpha}=F_{ap}(\cdot,\omega)*\phi_{\alpha}$, for a.a. $\omega\in\mathcal{E}$, where $*$ is the convolution in the 
$L^1-$mean, and the equation follows from definition of $\mathcal{N}(\R^n)$ which implies that 
$F_{\mathcal{N}}(\cdot,\omega)*\phi_{\alpha}=0$, for a.a. $\omega\in\mathcal{E}$. Now, the fact that $\phi_{\alpha}$ is a net is an inconvenience 
to be overcome by reducing $\phi_{\alpha}$ to a sequence $\{\phi_j{\}}_{j\ge 1}$ which serves as well as an approximation of the unity for the family 
$\{F_{ap}(\cdot,\omega);\,\omega\in \mathcal{E}_*\}$, where $\mathcal{E}_*\subset\mathcal{E}$ and $\mathbb{P}(\mathcal{E}_*)=1$. The way to achieve 
this reduction of $\phi_{\alpha}$ to a sequence $\phi_j$ in \cite{Frid} was to introduce a topology in $\mathcal{E}$ as a dense subset of a separable compact 
space so that the family $\{F(x,\cdot);\,x\in{\mathbb{Q}^n}\}$ generates the topology given to $\mathcal{E}$. This allows us to take a countable dense subset 
$D\subset\mathcal{E}$ and then we consider the separable closed subalgebra $\mathcal{A}_*$ of $\AP(\R^n)$ generated by the unit and 
$\{F_{ap}(\cdot,\omega);\,\omega\in D\}$, for which we may obtain from $\phi_{\alpha}$ from $\phi_{\alpha}$ a sequence $\phi_j$ which is an approximation 
of the identity for the whole family $\{F_{ap}(\cdot,\omega);\,\omega\in\mathcal{E}_*\}$. Indeed, one can consider the compact space associated with 
$\mathcal{A}_*$ by stone's theorem(see, e.g., \cite{DS}), which is then, after passing to a quotient space if necessary, a group for which there is an 
approximate identity sequence $\phi_n$ that may be seen as a subsequence of $\phi_{\alpha}$. The final step is to prove that 
$F(y,\omega)=F_{ap}(y,\omega)$ for a.a. $\omega\in\mathcal{E}$ and all $y\in\R^n$, which follows from the decomposition of 
$F(\cdot,\omega)=F_{ap}(\cdot,\omega)+F_{\mathcal{N}}(\cdot,\omega)$ and Birkhoff theorem, namely, that for each $y\in\R^n$, we have, denoting 
$M(g)$ the mean value of $g$, when $g:\R^n\to\R$ possesses mean value, 
\begin{eqnarray*}
&&\int_{\mathcal{E}}|{\gamma}(\mathcal{T}(y)\omega)-\tilde{\gamma}(\mathcal{T}(y)\omega)|\,d\mathbb{P}(\omega)=
\int_{\mathcal{E}}|{\gamma}(\omega)-\tilde{\gamma}(\omega)|\,d\mathbb{P}(\omega)\\
&&\qquad=M\Big(|{\gamma}(\mathcal{T}(\cdot)\omega_*)-\tilde{\gamma}(\mathcal{T}(\cdot)\omega_*)|\Big)=0,
\quad\text{for $\mathbb{P}-$a.a. $\omega_*\in\mathcal{E}$},
\end{eqnarray*}
by the invariance of $\mathbb{P}$ with respect to $\mathcal{T}(y)$, Birkhoff's relation and by the fact that 
$F(\cdot,\omega)\in\mathcal{W}^*\AP(\R^n)$.
\end{proof}

\begin{lemma}\label{UnifCont}
Define $F:\R\times\Omega\to\R$ by $F(x,\omega):=f(T(x)\omega)$, where the function 
$f(\omega)=\omega(0)$ and the dynamical system is such that $T(x)\omega=\omega(\cdot+x)$. Then, the following property holds:
$$
\lim_{s\to 0}\sup_{|x-z|<s}\bigg(\sup_{\omega\in\Omega}|F(x,\omega)-F(z,\omega)|\bigg)=0.
$$
\end{lemma}
\begin{proof}
1. Let $\varphi\in C_c(\R)$ be the function used in definition of the set $\Omega$. Remember that by 
definition of the set $\Omega$, given $\omega\in\Omega$ there exists a sequence 
$\{\mathbf{x}_m{\}}_{m\in\mathbb{Z}}\subset\{-1,1\}$ and $\delta\in\R$ such that 
$$
F(\cdot,\omega)=f(T(\cdot)\omega)=\omega(\cdot)=\sum_{m\in\mathbb{Z}}\mathbf{x}_m\varphi(\cdot+\delta-m).
$$

Since $\varphi$ is 
uniformly continuous, given $\epsilon>0$ there exists $t_0>0$ such that 
$$
|x-z|<t_0\Rightarrow |\varphi(x)-\varphi(z)|<\epsilon.
$$
Now, define $t_1:=1/2\,\mathrm{dist}\left(\supp \varphi,\partial (-1/2,1/2)\right)$ and 
take $s_0:=\min\{t_0,t_1\}$. Let $x,z\in\R$ be such that $|x-z|<s_0$. Since 
$\R=\cup_{m\in\mathbb{Z}}\left(-1/2,1/2\right]+m$, we have two possibilities:
\begin{itemize}
\item There exists $m_0\in\mathbb{Z}$ such that $x+\delta,z+\delta\in\left(-1/2,1/2\right]+m_0$. In this case,
$F(x,\omega)=\mathbf{x}_{m_0}\varphi(x+\delta-m_0)$ and $F(z,\omega)=\mathbf{x}_{m_0}\varphi(z+\delta-m_0)$, which implies 
$$
|F(x,\omega)-F(z,\omega)|<\epsilon.
$$
\item There exists $m_0\in\mathbb{Z}$ and $l_0\in\mathbb{Z}$ such that 
$x+\delta\in\left(-1/2,1/2\right]+m_0$, $z+\delta\in\left(-1/2,1/2\right]+l_0$ and 
$x+\delta-m_0,z+\delta-l_0\notin\supp\varphi$. Therefore,
$F(x,\omega)=0=F(z,\omega)=0$.
\end{itemize}
In any case,
$$
|x-z|<s_0\Rightarrow |F(x,\omega)-F(z,\omega)|<\epsilon\qquad\text{for all $\omega\in\Omega$}.
$$
 \end{proof}

Now, we are ready to prove the main result of this paper. 
\begin{theorem}\label{MainResult}
Let $\Big(\Omega,\mu\Big)\subset\BUC(\R)$ be the probability space introduced in the Lemma \ref{R1.1}. Then, for $\mu-$a.a. $\omega\in\Omega$, 
$$
\omega(\cdot)\notin\mathcal{W}^*\AP(\R).
$$
\end{theorem}
\begin{proof}
1. Let $\Big(\Omega,\mu\Big)$ be the probability space and $T:\R\times\Omega\to\Omega$ be the dynamical systems considered in the Lemma \ref{R1.1}. 
Define $\mathcal{F}:=\Big\{\omega\in\Omega;\,\omega(\cdot)\in \mathcal{W}^{*}\AP(\R)\Big\}$. Our aim is to show that 
$\mu(\mathcal{F})=0$. But, the lack of measurability of the 
set $\mathcal{F}$ is a problem. We overcome this lack of measurability of the set 
$\mathcal{F}$ by showing the existence of an invariant and measurable set $\mathcal{F}^{+}$ such that $\mathcal{F}\subset\mathcal{F}^{+}$. 
Then, we use the Theorem~\ref{PaperHermano} to conclude that $\mu(\mathcal{F}^{+})=0$ and so finish the proof of the theorem.

2. Let $E\subset\Omega$(not necessarily measurable) be such that $T(x)E=E$ for all $x\in\R$. Due to this invariance and the identification 
given by the Lemma \ref{R1.1} item $(i)$, we have that $H(E)$ is 
invariant by the dynamical system $S:\R\times \Omega_1\to\Omega_1$. 
This implies that we can find a set $C_0\subset\Omega_0$(not necessarily measurable) 
such that $H(E)=C_0\times [0,1)$. Moreover, we claim that 
$C_0\subset\tau_k(C_0)$ for all $k\in\mathbb{Z}$, where $\{\tau_k{\}}_{k\in\mathbb{Z}}$ is the 
shift operator acting in $\Omega_0$. The claim can be proved by reasoning in the following way: a sequence $\mathbf{x}\in C_0$. Due to 
the invariance of the set $C_0\times[0,1)$ by the dynamical system $S$, we must have 
$(\mathbf{x},0)\in S(k)\Big(C_0\times [0,1)\Big)$. Thus, there exists a sequence $\mathbf{y}\in C_0$ and 
$\delta\in[0,1)$ such that 
$$
(\mathbf{x},0)=S(k)(\mathbf{y},\delta)=
\Big({\tau}_{\left\lfloor k\,+\delta \right\rfloor} \mathbf{y},k+\delta- \left\lfloor k\,+\delta \right\rfloor\Big)=
\Big(\tau_k\mathbf{y},\delta\Big),
$$
which gives that $\mathbf{x}=\tau_k\mathbf{y}\in \tau_k(C_0)$. This proves the claim. Taking into account 
the arbitrariness of $k\in\mathbb{Z}$ in the claim, we have $C_0\subset \tau_{-k}(C_0)$. As a consequence, 
$\tau_k(C_0)=C_0$ for all integer $k$. 

3. Now, observe that there exists a measurable set $C_1\subset\Omega_0$ such that $C_0\subset C_1$ and 
$$
\nu(C_1)=\inf\Big\{\nu(C);\,\text{$C\subset\Omega_0$ is measurable and $C_0\subset C$}\Big\}.
$$
Since the set $C_0$ is $\tau-$invariant, we have that $C_0=\tau_k(C_0)\subset \tau_k(C_1)$ for all integer $k$. 
Hence, $C_0\subset \cap_{k\in\mathbb{Z}}\tau_k(C_1)=:C_0^{+}$. Moreover, we can see that $C_0^{+}$ is an 
$\tau-$invariant set and $\nu(C_0^{+})=\nu(C_1)$. Therefore, the set 
$\Big(C_0^{+},[0,1)\Big)$ is invariant by the dynamical system $S$, that is, 
$$
S(z)\Big(C_0^{+},[0,1)\Big)=\Big(C_0^{+},[0,1)\Big),
$$
for all $z\in\R$. Consequently, we have the existence of a measurable set $E^{+}\subset\Omega$ such that 
\begin{enumerate}
\item[(P1)] $E\subset E^{+}:=H^{-1}\Big(C_0^{+},[0,1)\Big)$.
\item[(P2)] $T(z)(E^{+})=E^{+}$ for all $z\in\R$.
\item[(P3)]
$$
\mu(E^{+})=\inf\Big\{\mu(C);\,\text{$C\subset\Omega$ is measurable and $E\subset C$}\Big\}.
$$
\end{enumerate}
Considering the ergodicity of the dynamical system $T$, we must have $\mu(E^{+})\in\{0,1\}$. 

4. Let $\mathcal{F}$ be as in the step $1$. Since the set 
$\mathcal{F}$ is invariant by the dynamical system $T$, we can apply the step $3$ for $E=\mathcal{F}$ 
and obtain the existence of a measurable set $\mathcal{F}^{+}\subset\Omega$ having the properties 
P1,P2 and P3 above. Furthermore, we have that $\mu(\mathcal{F}^{+})\in\{0,1\}$. We claim that $\mu(\mathcal{F}^ {+})=0$. 
Suppose that the opposite happens, that is, $\mu(\mathcal{F}^{+})=1$. In this case, we can endow the set $\mathcal{F}$ with 
the following probability structure: Consider the $\sigma-$algebra 
$$
\mathcal{A}:=\Big\{E\subset\mathcal{F};\,\text{$E=A\cap \mathcal{F}$ for some measurable set $A\subset\Omega$}\Big\}
$$
and define the probability measure $\mu^{+}:\mathcal{A}\to[0,1]$ as $\mu^{+}(E):=\mu(A\cap \mathcal{F}^{+})$. Also, 
define the mapping $\mathcal{T}:\R\times\mathcal{F}\to\mathcal{F}$ by $\mathcal{T}(x)\omega:=T(x)\omega$. It is clear 
that $\mathcal{T}(0)=Id$ and $\mathcal{T}(x+y)=\mathcal{T}(x)\circ \mathcal{T}(y)$. If $E\in\mathcal{A}$, then, 
$E=A\cap\mathcal{F}$ for some measurable set $A\subset\Omega$. Hence, due to the 
invariance of $\mathcal{F}$ by $T$, we have $\mathcal{T}(x)E=T(x)A\cap\mathcal{F}$. Since $T$ is a dynamical system on 
$\Omega$, we have $T(x)A$ is a measurable set. Consequently, $\mathcal{T}(x)(E)\in\mathcal{A}$ and due to the invariance 
of the set $\mathcal{F}^{+}$ by $T$, we get
$$
\mu^{+}\Big(\mathcal{T}(E)\Big):=\mu\Big(T(x)A\cap\mathcal{F}^{+}\Big)=\mu\Big(T(x)\left(A\cap\mathcal{F}^{+}\right)\Big)
=\mu\Big(A\cap\mathcal{F}^{+}\Big)=\mu^{+}(E).
$$
Therefore, $\mathcal{T}$ is a dynamical system acting on the probability space $\Big(\mathcal{F},\mathcal{A},\mu^{+}\Big)$. 
It remains to show the ergodicity of $\mathcal{T}$. For this, let $E\in\mathcal{A}$ be such that 
$\mathcal{T}(x)(E)=E$ for all $x\in\R$. By definition of $\mathcal{T}$, this means $T(x)E=E$ for all $x\in\R$. Using the step 
$3$, we can find a measurable set $E^{+}\subset\Omega$ satisfying the properties P1, P2 and P3. Moreover, 
$\mu(E^{+})\in\{0,1\}$. Suppose that $\mu(E^{+})=1$ and write $E=A\cap\mathcal{F}$. By the property P3, we have 
$$
1=\mu(E^{+})=\inf\Big\{\mu(C);\,\text{$C\subset\Omega$ is measurable and $E\subset C$}\Big\}\le 
\mu\Big(A\cap\mathcal{F}^{+}\Big)\le 1,
$$
for $E=A\cap\mathcal{F}\subset A\cap \mathcal{F}^{+}$. Therefore, $\mu^{+}(E):=\mu\Big(A\cap\mathcal{F}^{+}\Big)=1$. 

It is clear that if $\mu(E^{+})=0$ then $\mu^{+}(E)=0$. Thus, $\mathcal{T}$ is ergodic. 

Define $f:\Omega\to\R$ by $f(\omega)=\omega(0)$. By the  Lemma \ref{UnifCont}, the function 
$F(x,\omega)=f(\mathcal{T}(x)\omega)$ satisfies the hypotheses of the Theorem \ref{PaperHermano}. 
Consequently, $\mu^{+}\Big(\left\{\omega\in\mathcal{F};\,\omega(\cdot)\in\AP(\R)\right\}\Big)=1$, 
which is a contradiction with the Lemma \ref{R3}. Thus, we must have $\mu(\mathcal{F}^{+})=0$ and the 
theorem is proved.
\end{proof}

\section*{Acknowledgements}
The author acknowledges GOD for the inspiration. He also thanks the many helpful suggestions of H. Frid during the preparation of the paper and 
the support from CNPq, through grant proc.~302331/2017-4.

\end{document}